\documentclass[12pt]{amsart}%journal class file, must be downloaded and present in same directory
\usepackage{diagrams}%%
\usepackage{fullpage}
\usepackage{microtype}%%
\newtheorem{theorem}{Theorem}[section]
\newtheorem{lemma}[theorem]{Lemma}
\newtheorem{corollary}[theorem]{Corollary}
\newtheorem{proposition}[theorem]{Proposition}
\numberwithin{equation}{section}
\theoremstyle{definition}
\newtheorem{definition}[theorem]{Definition}
\theoremstyle{remark}
\newtheorem{remark}[theorem]{Remark}

\newtheorem{question}[theorem]{Question}

\newcommand{\arrow}{\rightarrow}
\newcommand{\tensor}{\otimes}

\renewcommand{\L}{\mathcal L}%%
\newcommand{\algtensor}{\odot}

\newcommand{\betatensor}{{\tensor_\beta}}
\newcommand{\AbA}{A_1 \betatensor A_2}
\newcommand{\dd}{\ensuremath{^{**}}}%%

\newcommand{\AbB}{A \betatensor B}

\newcommand{\Z}{\mathcal Z}
\newcommand{\Hb}{H_\beta}
\newcommand{\closure}[1]{\overline{#1}}

\newcommand{\T}{\ensuremath{\mathcal T}}
\newcommand{\compose}{\circ}%%

\newcommand{\lsc}{lower semi-continuous }
\newcommand{\Cu}{\ensuremath{{\mathcal C}\hspace{-.8pt}u}}
\newcommand{\leqCu}{\preceq_{\scriptscriptstyle {\mathcal C}\hspace{-.75pt}u}\!}
\newcommand{\geqCu}{\succeq_{\scriptscriptstyle {\mathcal C}\hspace{-.75pt}u}\!}%%
\newcommand{\pos}{{_{+}}}%%
\newcommand{\cp}{\sim_{\scriptscriptstyle {cp}\!}}
\newcommand{\ce}{\sim_{\scriptscriptstyle {\mathcal C}\hspace{-.75pt}u}\!}
\newcommand{\Id}{\mbox{Id}}
\newcommand{\blackadar}{\sim_{s}} %%
\newcommand{\cpoz}{orthogonality preserving completely positive  map }
\newcommand{\cpozs}{orthogonality preserving completely positive  maps }
\newcommand{\Ped}{\ensuremath{K_0^+}}%%
\newcommand{\seminorm}{faithful finite invariant convex  functional } 
\newcommand{\seminorms}{faithful finite invariant convex functionals }
\newcommand{\compact}{\ensuremath{\mathcal K}}
\newcommand{\zbl}[1]{}
\newcommand{\mrev}[1]{} 
\begin{document}

\title[The tensor product of Cuntz semigroups]{Traces and Pedersen
ideals of tensor products of nonunital C*-algebras,}
\author{Cristian Ivanescu}
\author{Dan Ku\v cerovsk\'y}
\address[1]{Department of Mathematics and Statistics, Grant MacEwan University, Edmonton, Alberta, T5J 4S2, e-mail:  ivanescu@macewan.ca}
\address[2]{Department of Mathematics and Statistics, University of New Brunswick, Fredericton, New Brunswick, E3B 5A3, e-mail: dkucerov@unb.ca}


\commby{}

\begin{abstract} We show that positive elements of a Pedersen ideal
of a tensor product can be approximated in a particularly strong
sense by sums of tensor products of positive elements. This has a
range of applications to the structure of tracial cones and related
topics, such as the Cuntz-Pedersen space or the Cuntz semigroup.
For example, we determine the cone of lower semicontinuous traces
of a tensor product in terms of the traces of the tensor factors, in
an arbitrary C*-tensor norm. We show that the positive elements
of a Pedersen ideal are sometimes stable under Cuntz equivalence.
We generalize a result of Pedersen's by showing that certain classes
of completely positive maps take a Pedersen ideal into a Pedersen
ideal. We provide theorems that in many cases compute the Cuntz
semigroup of a tensor product.
\end{abstract}

\maketitle


\section{Introduction}

A unital $C^*$-algebra is a noncommutative generalization of the algebra $C(X)$ of continuous functions on a compact topological space; in the nonunital case, it is a generalization of the algebra $C_0 (X)$ of continuous functions that vanish at infinity, or more accurately, functions that are each arbitrarily small outside a sufficient large compact set.  
This paper studies the structure and applications of the Pedersen ideals of tensor products of
nonunital $C^*$-algebras. The Pedersen ideal of a nonunital $C^*$-algebra is a noncommutative analogue of the
space of compactly supported functions in the space of continuous complex-valued functions on a locally compact
Hausdorff space. One would therefore expect the Pedersen ideal of the tensor product of two $C^*$-algebras to
“multiply” in an obvious sense. However, while this might be expected, the situation is complicated by the
fact that there is, in general, more than one $C^*$-norm that one can put on the tensor product of $C^*$-algebras.
Our theorem \ref{prop:pedtensor} shows that the Pedersen ideals nevertheless multiply in a very strong sense.
 Since there is more than one possible norm on the tensor product in this so-called non-nuclear case, it would be further  expected that the topology that Pedersen found is natural for Pedersen ideals  behaves in a complicated way under tensor products.  However, we find that there are at least some quite good properties under approximation by elements of the Pedersen ideals of the tensor factors, even without assuming nuclearity. More specifically,
we show a technical property of the Pedersen ideal of a tensor product;  that positive elements of the Pedersen ideal can be approximated by elements from the Pedersen ideals of the tensor factors that have a particularly strong form of positivity. See Corollary  \ref{cor:ped.approx}. We investigate applications to maps of Pedersen ideals (Corollary \ref{cor:slice.maps} and Proposition \ref{prop:tracial.slices}) and to the Cuntz semigroup, using the theory of (usually densely finite) \lsc traces on C*-algebras. We determine the cone of \lsc traces on a tensor product in terms of the tracial cones of the tensor factor, see corollary \ref{cor:TChomeo} and theorem \ref{th:LTC}

Pedersen introduced the minimal dense ideal that bears his name \cite[th. 1.3]{Pedersen1966}, and found numerous properties of it in a series of papers \cite{Pedersen1966,Pedersen1968,Pedersen1969}. In particular, the Pedersen ideal provides a common core for the domain of the densely finite \lsc traces. Cuntz and Pedersen \cite{CP} found a real Banach space, which we refer to as the Cuntz-Pedersen space, upon which traces are represented as positive linear functionals.
See corollary \ref{cor:CP1} for our main result on Cuntz-Pedersen spaces and  tensor products.

Effros introduced  a topology for the faces of the cone of \lsc traces \cite{Effros}. Perdrizet \cite{Perdrizet1970} and Davis \cite{Davies1969} separately studied the cone of densely finite \lsc traces, with the weak topology coming from the Pedersen ideal. They showed the cone is closed; they also showed that it is well-capped in the separable case (implying a theorem of Krein-Milman type);  and they studied properties of the faces. Elliott, Robert, and Santiago \cite{ERS} compactified the cone of densely defined \lsc traces, within a larger cone of  traces, and investigated applications to the Cuntz semigroup. In general, the larger cone of all traces on a C*-algebra is not cancellative, implying that it doesn't embed as bounded elements into any vector space. This apparent technicality has subtle consequences when one tries to study such traces. We will focus here on densely finite  \lsc traces. This is because of the aforementioned fundamental relationship between Pedersen ideals and such traces: namely, that the Pedersen ideal provides a common core for all such linear functionals, and  densely finite \lsc traces become continuous linear functionals when restricted to the Pedersen ideal.

The Cuntz semigroup has become a standard technical tool for the study of C*-algebras. The Cuntz semigroup $\Cu(A)$ can be defined to be an ordered semigroup with elements given by the positive elements of $A\tensor\compact,$ modulo an equivalence relation \cite{ORT}. The projection-class elements  are those elements that are equivalent to projections of $A\tensor\compact.$
The projection-class elements of the Cuntz semigroup of a tensor product C*-algebra are in good cases determined by the classical K\" unneth sequence \cite{Schochet1982} from $K$-theory. The remaining elements are called the purely positive elements, and it seems interesting to determine how the purely positive elements  are related to the Cuntz semigroups of the tensor factors. See corollary \ref{cor:useAPT} and theorem \ref{th:ideals}. In good cases, the subsemigroup of purely positive elements is in an appropriate sense a tensor product of the subsemigroups of purely positive elements of the tensor factors. Since the projection-class elements are already well understood (see Appendix A) this means that in good cases we can compute the Cuntz semigroup of a tensor product. We do not assume nuclearity.

In section 2, we study finite traces. In section 3 we consider the Pedersen ideal of a tensor product and give some applications. In section 4, we consider Cuntz semigroups and Cuntz-Pedersen spaces.

In the final section, we consider traces which are densely finite  on closed ideals of a given C*-algebra. There are two appendices. Appendix A is on the K\" unneth sequence for projection-class elements, and Appendix B contains some suggestions for extending the first steps of Grothendieck's  Memoire \cite{grot} to the case of Cuntz semigroups.

Namioka and Phelps \cite{NamiokaPhelps} define several different tensor products for compact convex sets. The different constructions coincide in that the tensor product of extreme points is extreme in the product. It may be more informative to make use of  tensor products of dual objects instead.  Thus, in the early paper \cite{arxiv}, we defined a tensor product of Cuntz semigroups by first representing them as a class of affine functions on the tracial simplexes and then declaring that the tensor product is given by a similar class of bi-affine functions on the two simplexes. See \cite[top of p. 475]{NamiokaPhelps} \cite[p. 304]{Semadeni} for a similar approach to (projective) tensor products. In the current paper we prove first some results on tensor products of Pedersen ideals  (theorem \ref{prop:pedtensor}), deduce a dual result on the tensor product of tracial cones (theorem \ref{th:LTC}), and then by using a different flavour of duality, deduce a result on the tensor product of Cuntz-Pedersen spaces (corollary \ref{cor:CP1}).  We consider tensor products of Cuntz semigroups (corollaries \ref{cor:useAPT} and \ref{th:ideals}, see also the appendices). 

\section{Commuting representations and tensor products}

Suppose that we are considering a tensor product $C:=A_1 \otimes_\beta A_2$ with respect to some given tensor product C*-norm $\beta.$
Let $\pi$ be the universal representation of $C$ on a Hilbert space $\Hb.$
Guichardet, as part of his study of tensor product norms \cite{Guichardet1965,Guichardet1969},  shows that $\pi$ has associated with it a canonical pair of commuting representations $\pi_1\colon A_1 \arrow \L(\Hb)$ and $\pi_2\colon A_2 \arrow \L(\Hb)$ such that $\pi(a\algtensor b)=\pi_1(a)\pi_2(b)$ for all $a\in A_1,$ $b\in A_2.$

For a $C^*$-algebra $A$ with tracial state space $T(A)$, let $\partial_e T(A)$ be the set of extreme points of the convex set $T(A)$. The tensor product of linear functionals gives a tensor product of $T(A)$ and $T(B),$ denoted $T(A)\tensor T(B).$ See \cite[Appendix T]{WeggeOlsen} for more information on the tensor product of linear functionals.  The natural map from $T(A)\tensor T(B)$ to $T(A\tensor_\beta B)$ will be denoted $\mathcal T.$ The following result is known, see \cite[pg. 49]{Guichardet1969}, but we  give a more streamlined proof, which was kindly pointed out to us by Rob Archbold:

\begin{theorem}Let $C:=\AbA$ be a $C^*$-tensor product of C*-algebras $A_1$ and $A_2$, and suppose that $\tau$ is a tracial state of $C.$ The following conditions are equivalent.
\begin{enumerate}\item $\tau \in \partial_e T(C)$.
\item There exist unique $\tau_i\in \partial_e T(A_i)$ ($i=1,2$) such
that $\tau=\tau_1\betatensor\tau_2.$\end{enumerate}   \label{th:TracialStates}\end{theorem}
\begin{proof}

(1) $\Rightarrow$ (2) If $\tau$ is extremal, then it is factorial by \cite[cor. 6.8.6]{Dixmier}. Guichardet \cite[pg. 26-27]{Guichardet1969}  
shows that the restriction(s) by the above map(s) $\pi_i$ preserves factoriality. Thus, $\tau_1:=\tau\circ\pi_1\colon A\rightarrow \mathbb C$ is a positive extremal tracial functional, of norm not greater than 1.
From Guichardet's definitions, it follows that $\tau_1=\lim_{\lambda\in\Lambda}\tau(a\otimes b_\lambda)$ where $b_\lambda$ is the universal increasing approximate unit of $B,$ consisting of all positive elements of $B$ that have norm no greater than one. It follows from the fact that $\tau$ is a state, that the limit of $\tau_1,$ with respect to an approximate unit, is 1.  
Thus, for positive $b$ in $B$ of norm no greater than 1, we note that for all positive $a$ we have $\tau(a \otimes b)\leq \tau_1 (a).$  Thus, for each such $b$ there exists a non-negative scalar $t_2,$ such that   $\tau(a \otimes b)= t_2 \tau_1 (a).$ Clearly $t_2$ depends on $b,$ say $t_2=\tau_2 (b).$  
Considering an approximate unit for $A,$ we see that $\tau_2(b)=\lim_{\lambda\in\Lambda}\tau(a_\lambda \otimes b).$ Thus, $\tau_2$ is the restriction of $\tau$ by $\pi_2,$ and as before is extremal.
Thus, $\tau=\tau_1\otimes \tau_2 $ is a tensor product of extremal tracial states.

Thus there exists at least one factorization  $\tau=\tau_1\tensor\tau_2$ of the required form. If there were a different such factorization, say $\tau=\tau'_1\tensor\tau'_2,$ then for at least one value of $i,$ there would exist elements of $A_i^+$ that separate $\tau_i$ from $\tau'_i.$ But then by tensoring these elements by an element of an approximate unit of the other tensor factor, we could separate  $\tau_1\tensor\tau_2$ from $\tau'_1\tensor\tau'_2$, giving a contradiction.

(2) $\Rightarrow$ (1). Suppose that $\tau=\tau_1\tensor\tau_2$ with $\tau_i\in \partial_e T(A_i)$ ($i=1,2$). By \cite[6.8.6]{Dixmier}, $\tau_i$ is a factorial state of finite type ($i=1,2$). Hence $\tau_1\tensor\tau_2$ is a factorial state \cite[p. 25]{Guichardet1969}, necessarily of finite type since the state is tracial. By \cite[6.8.6]{Dixmier}, $\tau$ is extreme.
\end{proof}

Denote by $TC$ the cone of finite traces of a C*-algebra.

\begin{corollary} Let $C:=\AbA$ be a C*-tensor product of C*-algebras $A_1$ and $A_2.$ Then $\T$ is a homeomorphism from $TC(A_1)\tensor TC(A_2)$ onto $TC(C).$  \label{cor:TChomeo}\end{corollary}
\begin{proof} The map $\T$ is injective on algebraic elements,\footnote{see for example Corollary 1.3.5 in \cite{DNR}.}
and the domain has the induced norm $\|\T(\cdot)\|;$ in other words, the map is injective by construction. %algebraic injectivity is something I learned about in studying quantum groups, thus my reference is the Hopf algebras book by DNR

 The cone of finite traces on a C*-algebra is the weak* closure of the convex hull of its extremal rays and extremal points. This is because the unit ball provides a cap, in Choquet's sense \cite{Choquet}, for the cone of finite traces, and then there is a theorem of Krein-Milman type, see Theorem 2.2 in \cite{Asimov1969}. Theorem \ref{th:TracialStates} implies that the range of the map $\T\colon TC(A_1)\tensor TC(A_2)\arrow TC(C)$ contains the extreme points of the range.  A standard scaling argument and the fact that the tracial cone intersected with the closed unit ball is  weak* compact \cite[prop. 6.8.7]{Dixmier}, shows that the range of the map $\T$ is closed, and hence that $\T$ is surjective.  In fact, a similar argument shows that by restricting the domain to the compact Hausdorff space given by intersecting $TC(A_1)\tensor TC(A_2)$ with a multiple of the unit ball, the restricted map $\T$ becomes a homeomorphism onto its range. Thus, again by a scaling argument, $\T$ is a homeomorphism.
\end{proof}

\section{Pedersen ideal of a tensor product }

Pedersen's ideal was originally defined as a minimal dense ideal, and is a common core for the \lsc densely finite traces. Pedersen uncovered various properties and characterizations of this ideal in a remarkable series of papers \cite{Pedersen1966,Pedersen1968,Pedersen1968b,Pedersen1969}. We recall that Pedersen's ideal, denoted by him $K(A),$ is  generated in a suitable sense \cite[p.133]{Pedersen1966}  by  $\Ped (A):=\{a\in A^{+} : \exists b\in A^+, \, [a]\leq b\}$, where $[a]$ denotes the range projection of $a$ in $A^{**}.$ We should mention that  the  set $\Ped(A)$ is generally not the same as the positive cone $K(A)\pos$ of $K(A).$ 
Recall that  Pedersen ideals are \cite[Theorem 2.1]{Pedersen1968b} locally convex topological vector spaces with a topology denoted $\tau.$ There also exists a weak topology. The underlying $\tau$-topology is stronger than the norm topology coming from the ambient C*-algebra. Thus density with respect to the $\tau$-topology, as provided by the following $\tau$-density theorem, is a very strong property.  

We denote by $\algtensor$ the algebraic tensor product. Then 
\begin{theorem}[$\tau$-density]Let $A_1$ and $A_2$ be C*-algebras. Then \begin{enumerate}\item  $\Ped(\AbA)\supseteq \Ped(A_1)\tensor\Ped(A_2),$\item $K(\AbA)\supseteq K(A_1)\tensor K(A_2),$ 
\item   $K(A_1)\algtensor K(A_2)$ is $\tau$-dense in $K(\AbA),$ and
\item  $K(A_1)_{+}\algtensor K(A_2)_{+}$ is $\tau$-dense in $K(\AbA)_{+},$\end{enumerate}
\label{prop:pedtensor}\end{theorem}
\begin{proof}
(1)  Let $a$ be an element, with norm not greater than one, of $\Ped(A_1),$ having range projection majorized by $a_c$ and let $b$ be an element, with norm not greater than one, of $\Ped(A_2),$ having range projection majorized by $b_c.$  Let $\pi_1$ and $\pi_2$ be the homomorphisms that embed $A_i$ in the universal representation of $\AbA.$  Since majorization of the range projection implies that $\pi_1(a^{1/n})\leq\pi_1(a_c)$ and $\pi_2(b^{1/n})\leq\pi_2(b_c),$ we apply $\pi_2(b_c)^{1/2}$ from left and right in the first inequality, then use commutativity and the second inequality, to conclude
$					\pi_1(a^{1/{n}})\pi_2(b^{1/n})\leq
					\pi_1(a_c)\pi_2(b_c).$
Taking the strong limit in $(\AbA)\dd$ as $n\arrow\infty$ shows that the range projection of $a\tensor b$ is majorized by  $\pi_1(a_c)\pi_2(b_c).$
This proves (1).  Taking linear combinations proves (2).

(4) We first show that $C\pos\pos:=K(A_1)\pos \algtensor K(A_2)\pos$ is weakly dense in $K(\AbA)\pos.$ 
A positive continuous linear functional, $f(x)$ on $\Ped(\AbA)\pos,$ extends to a \lsc function on the ambient C*-algebra, $\AbA,$ given by \cite{Pedersen1966} the pointwise supremum  
of the set of positive linear functionals $ \{ g\in {(\AbA)^{*}}\pos |\, g(x)\leq f(x), x\in K(\AbA)\pos\}.$ Each bounded positive functional in the set is determined by its values on $C\pos\pos,$ which is a linearly norm-dense subset of the C*-algebra, $\AbA.$ Thus  $f$ is determined by $\sup\{ g\in {(\AbA)^{*}}_{+} |\, g(x)\leq f(x), x\in C\pos\pos\}.$
It follows that a positive linear functional is determined by its values on $C\pos\pos.$

 We show that the same is true for self-adjoint continuous linear functionals. Supposing that a self-adjoint linear functional had two such extensions, then we would have a self-adjoint linear functional $h$ that is zero on $C\pos\pos$ but isn't zero. Pedersen's Jordan decomposition \cite[pg.127]{Pedersen1969b}, decomposes  $h$ uniquely into the difference of two positive continuous orthogonal linear functionals, $f_1$ and $f_2.$ It was shown in the previous paragraph that each continuous  positive linear functional is determined by its values on  $C\pos\pos.$ Since by, hypothesis, $f_1$ and $f_2$ become equal when restricted to $C\pos\pos,$ it follows that actually they are equal everywhere, and therefore the given self-adjoint linear functional $h$ must be zero everywhere.
 Thus, self-adjoint continuous linear functionals in fact extend uniquely from $C\pos\pos.$   
Then, by Hahn-Banach,  the convex set $C\pos\pos$ is weakly dense in  $K(\AbA)\pos.$ The celebrated Mazur theorem for locally convex topological vector spaces \cite[Cor. 3.46]{FHHMZ} implies density with respect to the topology $\tau.$

To prove (3),  if there exists some element $x$ in $K(\AbA)$ that is not a Moore-Smith limit of elements of $C,$  the usual decomposition of an element into a linear combination of four positive elements, which still holds in a Pedersen ideal, shows that there is a positive element  $x$ in $K(\AbA)\pos$ that is not a Moore-Smith limit of elements of $C.$ But this contradicts (4).
\end{proof}

The above key result will be very useful. The first application will be to computing the cone of \lsc traces of a tensor product. 

Before proceeding with this, there is still one more result that can be drawn from the same well as   theorem \ref{prop:pedtensor} above. Simple tensors of positive elements are sometimes called super-positive elements, see \cite[p.5]{Ando2004}. The next corollary provides approximations in the $\tau$-topology by such sums. It is perhaps unexpected that we can approximate from below by such elements.
\begin{corollary}Let $x$ be a nonzero positive element of the Pedersen ideal $K(\AbB).$ We can approximate $x$ from below, in the $\tau$-topology, by finite sums $\sum a_i\tensor b_i$ where the $a_i$  are positive elements of $K(A)$ and the $b_i$ are positive elements of $K(B).$ \label{cor:ped.approx}\end{corollary}
\begin{proof}  Recalling that the $\tau$-topology on $K(\AbB)$ is hereditarily convex \cite[pg.67]{Pedersen1968b}, we then use Theorem \ref{prop:pedtensor}.(4) to approximate $x$ from below by elements of $K(A_1)_{+}\algtensor K(A_2)_{+}.$ \end{proof}

\subsection{On \lsc tracial cones of tensor products}
A \lsc trace on a tensor product $\AbB$ is determined by its values on the Pedersen ideal. From theorem \ref{prop:pedtensor} it follows that in fact it is sufficient, for any norm $\beta,$ to know what the values are on the elementary tensors $k_1\otimes k_2$ where the tensor factors belong to the positive part(s) of the Pedersen ideals $K(A)_+$ and $K(B)_+.$ Since after all the \lsc traces become bounded on the Pedersen ideal, it then follows from the case of bounded traces (Corollary \ref{cor:TChomeo}) that
$LTC(A)\tensor LTC(B)$ is homeomorphic to $LTC(\AbB),$ where $LTC$ denotes the cone of densely finite \lsc traces. For the reader wishing more detail, we provide lemma \ref{lem:istrace}, proposition \ref{prop:PedersenAlgebras}, and corollary \ref{cor:restriction}; then proving the claimed result in theorem \ref{th:LTC}.
\begin{lemma} The tensor product of \lsc traces on $A_1$ and $A_2$ is a \lsc trace on $C:=\AbA.$ Positive continuous linear functionals on  $K(A_1)\algtensor K(A_2)$ extend to positive continuous linear functionals on $K(\AbA)$
\label{lem:istrace}   \end{lemma}
\begin{proof}  We are given positive linear functionals $\tau_1$ and $\tau_2$ on $\Ped(A_1)\pos$ and  $\Ped(A_2)\pos.$ 
The linear functional $\tau_1(\cdot ) \tensor \tau_2(\cdot )$ defines a  positive linear functional $\tau$ on $\Ped(A_1)\pos\algtensor\Ped(A_2)\pos,$ but by Theorem 
\ref{prop:pedtensor}.(4) we can extend the domain to all of $\Ped(\AbA)\pos.$ The tracial property $\tau(ab)=\tau(ba)$  is straightforward to check. The second part is similar, using Theorem 
\ref{prop:pedtensor}.(3) and an extension theorem \cite[Theorem 3.31]{FHHMZ}.   
 \end{proof}

The next two known propositions will be needed. The first one is  \cite[Prop. 4]{Pedersen1968}, with a shorter proof than the original, see also \cite[Prop. 2.1]{PedPet}.

\begin{proposition} Let $a$ be an element of $K(A)$. The hereditary subalgebra generated by $a$ is contained in the Pedersen ideal.\label{prop:PedersenAlgebras}\end{proposition}
\begin{proof} By  Cohen's factorization theorem \cite{Cohen1959} \cite[p. 150--151]{Pedersen1998}, an element $b$ of the hereditary subalgebra generated by $a$ in a C*-algebra $A$ can be written as $f(a)a',$ where $f$ is a continuous non-negative function on the spectrum of $a$ with $f(0)=0$ and $f\leq1,$ and $a'$ is some element of $A$ that can be chosen arbitrarily close, in norm, to the given element $b.$  It follows from the spectral theorem and the fact that $f(0)=0$ that the range projection $[f(a)]$ of $f(a)$ is majorized by the range projection $[a]$ of $a.$ But then, $f(a)$ is in $\Ped(A)$ when $a$ is.  Since the Pedersen ideal is an algebraic two-sided ideal, see the end of Lemma 1.1 in \cite{Pedersen1966}, it follows that  the given element, $b=f(a)a'$ is in the Pedersen ideal, $K(A).$
\end{proof}

\begin{corollary} Let $A$ be a C*-algebra. Let $k$ be a element of $K_0^+ (A).$ Let $B$ be the hereditary subalgebra $\closure{kAk}.$ Then a \lsc trace $\tau$ of $A$ becomes a bounded trace when restricted to $B.$
\label{cor:restriction}\end{corollary}
\begin{proof}By Proposition \ref{prop:PedersenAlgebras}, an element $\tau$ of the cone of \lsc traces, $LTC(A),$ becomes a pointwise finite positive linear functional when restricted to $B.$ If this positive functional were not bounded on $B,$ then we could find a sequence $x_n$ in the positive unit ball, $B^+_1,$ such that $\tau(x_n)>n2^{-n}.$ But then $x:=\sum 2^{-n} x_n$ is in the positive unit ball of $B$ and $\tau(x)\geq 2^{-n} \tau(x_n)>n,$ for all $n,$ which is a contradiction. \end{proof}

Recall Pedersen showed \cite{Pedersen1968b} that Pedersen ideals are topological vector spaces in a suitable topology, and that densely finite lower semicontinuous traces $LTC({A_i})$ can be identified with continuous linear functions on Pedersen ideals. When we form a tensor product of C*-algebras, the Pedersen ideals inherit a tensor product,  and thus so do the continuous linear functionals on the Pedersen ideals.  
As mentioned before,   we have a canonical algebraic injection 
$\T\colon K(A)^* \algtensor K(B)^* \rightarrow (K(A)\algtensor K(B))^*,$ and this provides (lemma \ref{lem:istrace}) an injection into    $(K(A\betatensor B))^*.$
 We define $LTC(A_1)\betatensor LTC(A_2)$ to be the closure,  with respect to the Pedersen topology induced by evaluation on elements of $K(A_1 \betatensor A_2 ),$ of 
  $LTC(A_1)\algtensor LTC(A_2).$
 Let $A_1$ and $A_2$ be C*-algebras, and let $C:=\AbA$ be their tensor product. 
Let $\T$ be the natural mapping $\T\colon  LTC({A_1})\betatensor LTC({A_2})\arrow LTC(C)$ after the above identification. 

Next, we in effect show that  under this mapping, the cone of \lsc traces, $LTC(C),$ is equal to the closure of $\{\tau_1\tensor\tau_2 : \tau_i\in LTC(A_i)\},$ where $LTC(A_i)$ denotes the cone of \lsc traces on $A_i,$ and the topology is the weak topology $\sigma(LT(C),K(C))$  provided by the Pedersen ideal. More precisely, we show that there is a homeomorphism at the level of topological cones:

\begin{theorem} Let $C:=\AbA$ be a C*-tensor product of C*-algebras $A_1$ and $A_2.$ Then $\T$ is a homeomorphism from $LTC(A_1)\betatensor LTC(A_2)$ onto $LTC(C),$ where $LTC$ denotes the cone of densely finite \lsc traces. \label{th:LTC}
\end{theorem}
\begin{proof}
Denote by $\{k_\lambda\}$ and $\{k'_{\lambda'}\}$ the set of positive elements of norm less than 1 of $\Ped(A_1)$ and $\Ped(A_2)$ respectively. By \cite[th. 1.4.2 and par. 1.4.3]{Pedersen} each of these sets becomes a net, and approximate unit, under the usual partial ordering inherited from the ambient C*-algebra. By proposition \ref{prop:pedtensor},  the  tensor product $k_\lambda\tensor k'_{\lambda'}$, with index set given by the product order, is in $\Ped(C).$  By corollary \ref{cor:restriction},  the inclusion map
$$\iota\colon\closure{(k_\lambda\tensor k'_{\lambda'})C(k_\lambda\tensor k'_{\lambda'})}\arrow C$$ induces a restriction map
$$ \iota^{*}\colon LTC(C)\arrow TC(\closure{(k_\lambda\tensor k'_\lambda)C(k_\lambda\tensor k'_{\lambda'})}).$$
Similarly, there are inclusions, also denoted $\iota,$ of $\closure{k_{\lambda} A_1  k_{\lambda}}$ into $A_1$ and of $\closure{k'_{\lambda'} A_2 k'_{\lambda'} }$ into $A_2,$ and the associated restriction maps.

By \cite[Th. 3.12]{ERS}\footnote{Note that our $LTC(A)$ is denoted $T_A(A)$ in \cite{ERS}: we first apply their theorem 3.12 to get a projective limit decomposition of their $T(A)$, and then restrict using their proposition 3.11 to  $T_A(A).$}, dualize the trivial inductive limit(s) that have canonical homomorphisms given by the inclusion maps $\iota.$  This gives projective limit decompositions of the \lsc tracial cones:
\begin{diagram}
 TC(\closure{(k_\lambda\tensor k'_{\lambda'})C(k_\lambda\tensor k'_{\lambda'})}) &   & \lTo^\lim_{\iota^*} & LTC(C) \\
    \uTo^{\T}  &  &   &                                       \uTo_{\T} \\
  TC(\closure{k_\lambda A_1 k_\lambda})\betatensor TC(\closure{k'_{\lambda'} A_2 k'_{\lambda'}})  &   & \lTo^\lim_{\iota^*\tensor\iota^*\,}  &  LTC(A_1)\betatensor LTC(A_2 ) .\\ \end{diagram}
 Because of the simple form of the inclusion maps, it is clear that the map(s) in the left column are just restrictions of the map in the right column. The rows are projective limits, and the diagram commutes: in other words, the restrictions of the map $\T$ are coherent with respect to the projection maps. Corollary \ref{cor:TChomeo} shows that the vertical arrows in the left column are homeomorphisms. Proposition  2.5.10 in \cite{Engelking} shows, in effect, that the limit of homeomorphisms is a homeomorphism, and then by the universal property of projective limits, which ensures that the map provided by that proposition coincides with $\T,$ the map in the right column  is a homeomorphism.
\end{proof}

\begin{question} Does a result similar to the above hold for Pedersen's  \cite{Pedersen1971} C*-integrals? \end{question}
\section{Applications}

The Pedersen ideal is closely related to the Cuntz-Pedersen space \cite{CP}. Both are spaces upon which a class of traces act as bounded linear functionals. We are thus able to expand upon some of the  results of \cite{CP}, see Corollaries  \ref{cor:CP1} and \ref{cor:CP2}. In a similar vein,  dimension functions on a Cuntz semigroup are constructed in a slightly complicated way from the traces on the underlying  C*-algebra. Thus, we expect some, possibly subtle, connection between Pedersen ideals and Cuntz semigroups.

One fundamental comment on Pedersen ideals and Cuntz semigroups is that the Pedersen ideal is in fact the (smallest) ideal linearly generated by the relatively compact elements of the Cuntz semigroup (see a Remark without proof after 3.1 in \cite{TTk}). After demonstrating this fact, in Proposition \ref{prop:compactly.generated}, we then apply a Cuntz semigroup technique to prove a generalization of Pedersen's result \cite[Corollary 6]{Pedersen1968} that a surjective homomorphism maps the Pedersen ideal into and onto the  Pedersen ideal. We then further show that  the property of being a positive element of the Pedersen ideal is preserved under Cuntz equivalence in the stable rank 1 case. This may be surprising because the Pedersen ideal is after all not norm closed.

The original definition of the property  that $a$ is compact with respect to $b$ was phrased in terms of Hilbert modules \cite{CEI}. An equivalent definition is:
\begin{definition} Given two positive elements $a$ and $b$ of a C*-algebra $A$, we say that $a$ is \textit{compact with respect to} $b$ if some positive element $e$ of the hereditary subalgebra generated by $b$ acts as the unit on $a.$ Thus, $ae=a.$ We say that $a$ is \textit{compact} if it is compact with respect to itself.
\label{def:compact}
\end{definition} 
The above definition is implicit in \cite{CEI} but is given more explicitly in \cite[Prop. 4.10]{ORT}. For example, if $A$ is unital then every positive element is compact with respect to the unit.  

\begin{proposition} Elements that are compact with respect to some element of $A\tensor \compact$ are contained in $\Ped(A\tensor \compact).$ Furthermore, sums of these elements generate, as a hereditary subalgebra, $\Ped(A\tensor \compact).$ \label{prop:compactly.generated}
\end{proposition}
\begin{proof} Consider then a positive element $a$ of $A\tensor \compact$ such that some other element $a'$ in $A\tensor \compact$ acts as the unit on it. This means that $aa'=a.$ However, one of Pedersen's characterizations \cite{Pedersen1971} of the Pedersen ideal is as follows:

Let the set $K_{00}^+(A\tensor \compact)$ be $\{b\in A\tensor \compact : bb'=b  \mbox{ for some } b'\in A\tensor \compact \}.$ Let $K_1^+(A\tensor \compact)$ be the elements $x\in A\tensor \compact^+$ such that $x\leq \sum x_i$ for some finite sum of $x_i$ in $ K_{00}^+(A\tensor \compact).$  Then the linear span of $K_1^+(A\tensor \compact)$ is the Pedersen ideal. 
But comparing the definition of $K_{00}^+(A\tensor \compact)$  with the property described in the first paragraph, we see that the proposition follows.
\end{proof}

It has been suggested \cite{WinterZacharias} that the appropriate class of morphisms of the Cuntz semigroup are maps induced by orthogonality preserving completely positive  maps. The term completely positive order zero map is also used.  
Recall that the  \cpozs are equivalently cross-sections of certain homomorphisms. This result is simplest to state for contractions, but any \cpoz can be made into a contraction by multiplying by a scalar.  

\begin{lemma}[\protect{\cite[Cor. 4.1]{WinterZacharias}}] Let A and B be C*-algebras, and $\phi : A \to B$  a given contractive \cpoz. Let $f\in  C((0, 1])$ be the identity map, $f(x)=x.$ Then, there exists a *-homomorphism  $\rho \colon C_0 ((0, 1]) \tensor A \to B$ such that the slice map given by $\rho (f \tensor a) $ is equal to the given map $\phi(a).$ 
Also, any given *-homomorphism $\rho \colon C_0 ((0, 1]) \tensor A \to B$ induces a \cpoz $\phi : A \to B$ via $\phi(a) := \rho(f \tensor a).$
\label{lemma:factor.by.hom}
\end{lemma} 

We now show that  such maps take the Pedersen ideal into the Pedersen ideal. 
Recall that \cite[Theorem 2.3 and 2.4]{Pedersen1969} that the $\tau$-topology on a Pedersen ideal is given  by seminorms, and the seminorms are in fact the \seminorms on the Pedersen ideal. These seminorms can be extended to \lsc \seminorms on the positive cone of the ambient C*-algebra. Conversely, a \lsc \seminorm is finite on the Pedersen ideal.
For the reader's convenience we recall that invariant positive convex functionals have the properties:
\begin{enumerate}\item $\tau(\alpha x)=\alpha\tau(x),$ for positive real $\alpha,$ 
\item $\tau(x+y)\leq \tau(x)+\tau(y),$ for positive elements $x$ and $y$
\item For all $v$, $\tau(vv^*)=\tau(v^*v),$ and
\item If $x\geq y$ then $\tau(x)\geq\tau(y).$
\end{enumerate}

Now a simple but useful lemma:
\begin{lemma} If $\phi : A \to B$ is a faithful \cpoz, and $\tau$ is a \lsc \seminorm on $B,$ then $\tau\compose\phi$ is a \lsc \seminorm on $A.$ \label{lem:boundedness}\end{lemma}
\begin{proof} The only property that is not immediate is that $\tau(\phi(vv^*))=\tau(\phi(v^*v)).$ However, expanding the left side as sections of a homomorphism $\rho,$ giving $\tau(\rho(f^{1/2}\tensor v)\rho(f^{1/2}\tensor v)^*)$ and the right side similarly, the property follows.
\end{proof}%leave this, it doesn't get used but it seems to be interesting.

\begin{proposition} A *-homorphism $\phi\colon A\to B$ maps the Pedersen ideal of $A$ into the Pedersen ideal of $B.$ \label{prop:hom}  \end{proposition}
\begin{proof}
One of Pedersen's characterizations \cite{Pedersen1971} of the Pedersen ideal $K(A)$ of $A$ is that it is generated by $K_{00}^+ (A):=\{a\in A^{+} : \exists b\in A^+, ab=a\}$. A*-homomorphism will preserve the property $ab=a$ and thus maps $K_{00}^+ (A)$ into $K_{00}^+ (B).$ 
\end{proof}

Now we will show a result that is similar to Pedersen's result \cite[Corollary 6]{Pedersen1968} that a surjective homomorphism maps the Pedersen ideal into and onto the  Pedersen ideal.

\begin{theorem} A faithful surjective \cpoz $\phi\colon A\to B$ maps the Pedersen ideal of $A$ into the Pedersen ideal of $B,$ if $B$ is either simple, prime, or has finite-dimensional centroid $Z(\mathcal M(B)).$ 
\label{th:cpoz}\end{theorem}
\begin{proof}
Let us suppose that B has finite dimensional centroid $C :=
Z(\mathcal M(B)).$ Suppose that $a$ is some positive element of the Pedersen
ideal of A. The homomorphism provided by lemma \ref{lemma:factor.by.hom} has the property
\cite[Cor. 4.2] {WinterZacharias} that $\rho(g, a) = g(z)\pi(a)$ where $z$ commutes with and
multiplies the range of the given map $\phi$, and $\pi$ is a homomorphism.
Note that $z$ is thus in the centroid $C.$ Let $f \in C_0((0, 1])$ be the identity
map, $f(x) = x.$ Because the centroid is finite-dimensional, we can find
a function $g \in C_0((0, 1])$ such that $g(z)$ is a projection and $g$ acts as
the identity on $f.$ But then we have that $\phi(a) = \rho(f, a) = f(z)\rho(g, a)$
The map $a \mapsto \rho(g, a) = g(z)\pi(a)$ has range contained in $B$ by lemma \ref{lemma:factor.by.hom}, and is a homomorphism because the element $g(z)$ is a central projection and
because $\pi$ is a homomorphism. Thus this map takes the Pedersen
ideal of $A$ into the Pedersen ideal of $B.$ Thus $\rho(g, a)$ is an element of
the Pedersen ideal of $B,$ and thus when multiplied by $f(z)$ we get again
an element of the Pedersen ideal. Thus $\phi(a) = \rho(f, a) = f(z)\rho(g, a)$ is
in the Pedersen ideal. This proves the result in the case that B has finite dimensional centroid, and the other cases then follow from the
Dauns-Hoffman theorem \cite[Cor. 4.4.8]{Pedersen}.
 \end{proof} 

%If we drop faithfulness in the hypotheses of the above theorem, then the composition $\tau\compose \phi$ which appears in the proof need no longer be faithful, but it is still finite on elements of the Pedersen ideal, and this is all we really need. Thus, we have the corollary:
%
%\begin{corollary} A \cpoz $\phi\colon A\to B$ maps the Pedersen ideal of $A$ into the Pedersen ideal of $B.$
%\end{corollary}
%
%Since $C(0,1]$ is, up a scale factor, the universal C*-algebra generated by a positive element, from the associated *-homomorphism $\rho: C((0,1])\tensor A\to B$ we have the following less technical corollary:
\begin{corollary} Let $\rho\colon D\betatensor A\rightarrow B$ be a *-homomorphism. Let $f$ be
either a compact element of $D$ or an element of $D$ such that $0$ is isolated
in the spectrum of $|f|.$ The slice map $\rho\colon A\to B$ given by $a \mapsto \rho(f\tensor a)$ maps the Pedersen ideal of $A$ into the Pedersen ideal of $B.$
\label{cor:slice.maps}\end{corollary}
\begin{proof}
Suppose that that 0 is isolated in the spectrum of the given
element $|f|.$ Then $C^*(f)$ is unital. We may restrict the given homomorphism
to $\rho\colon C^*(f) \betatensor A \rightarrow B.$ Since $C^*(f)$  is unital, by Theorem 
\ref{prop:pedtensor}.(3) the Pedersen ideal of $C^*(f) \betatensor A$
contains $C^*(f) \betatensor K(A).$ 
Then by Proposition \ref{prop:hom} this is mapped by $\rho$ into the Pedersen ideal of
$B,$ and thus we have shown the first case. 
In the second case, namely that $f$ is compact, it follows from Definition \ref{def:compact} that the hereditary
subalgebra generated by $f$ in $D$ is unital. We then restrict the given
homomorphism to $\rho\colon \overline{fDf}\betatensor A \rightarrow B,$ 
and proceed similarly. \end{proof}
The above shows that techniques used to study Cuntz semigroups are also useful with the Pedersen ideal. 
To show that the orthogonality condition is not necessary, we give one more result:
\begin{proposition}Suppose that $t$ is a \lsc trace of $A.$ Then, the slice map $t\tensor \Id\colon A\betatensor B \rightarrow B$ maps the Pedersen ideal of $A\betatensor B$ into the Pedersen ideal of $B.$
\label{prop:tracial.slices}\end{proposition}
\begin{proof} The tensor product of a \lsc trace and a \lsc \seminorm is readily seen to be a \lsc \seminorm on $A\betatensor B.$ 
 But then if $w$ is a positive element of the Pedersen ideal of $A \betatensor B,$ so that corollary \ref{cor:ped.approx} provides approximations from below of the form 
$\sum_1^n a_i\tensor b_i,$ where the terms $a_i$ and $b_i$ come from the Pedersen ideals of the tensor factors, we note that when we apply a \seminorm of $B$ to the image $\sum_1^n t(a_i) b_i$ of this sequence under the slice map, we obtain a sequence that  is increasing and bounded above, hence convergent. Thus, the image $(t\tensor \Id) (w)$ of $w$ under the slice map is in the Pedersen ideal of $b.$ 
\end{proof}
The above slice map is completely positive but is not generally orthogonality-preserving, leading to a question:
\begin{question} What  class of completely positive maps $\phi\colon A\to B$ does in general map the Pedersen ideal of $A$ into the Pedersen ideal of $B$?
\end{question}
In some cases we can prove a very strong property: that the Pedersen ideal is itself stable under Cuntz equivalence.  What we prove precisely is that the property of being a positive element of the Pedersen ideal is preserved under Cuntz equivalence in the stable rank 1 case.  This case arises because then a normally finer equivalence relation, denoted $a\blackadar b,$ becomes the same as Cuntz equivalence. 
We need some preliminaries. 

We recall an equivalence relation on positive elements of a C*-algebra, $A,$ that was considered by Blackadar  \cite{blackadar88}, and is denoted here by $a\blackadar b.$ This equivalence relation is generated by the following two equivalence relations:
\begin{enumerate}
	\item positive elements $a$ and $b$ are equivalent in a C*-algebra $A$ if they generate the same hereditary subalgebra of $A,$ and
\item
positive elements $a$ and $b$ are equivalent in a C*-algebra $A$ if there is an element $x\in A$ such that $a=x^*x$ and $b=xx^*.$ \end{enumerate}

 We say that $a$ is Cuntz subequivalent to $b,$ denoted $a \leqCu b$, if there is a sequence $x_n$ such that $x_n^* b x_n$ goes to $a$ in norm. Thus, for example, $e^*xe\leqCu x.$ We say that $a$ and $b$ are Cuntz equivalent, and write $a\ce b,$ if we have both $a \leqCu b$ and $a \geqCu b.$ 

The following known Lemma relates the equivalence relation $\blackadar$ on positive elements to Cuntz equivalence and to properties of the Hilbert submodules of $A$ that are generated by the given positive elements:
\begin{lemma}Let $A$ be a C*-algebra, and let $a$ and $b$ be positive elements of $A.$ The following are equivalent:
\begin{enumerate}\item $a\blackadar b,$ and
\item $\overline{aA}$ and $\overline{bA}$ are isomorphic as Hilbert $A$-modules.

\item If  $A$ has stable rank 1, then the above is equivalent to  $a\ce b$ in the Cuntz semigroup. 
\end{enumerate}\label{lem:isos}
\end{lemma}
\begin{proof} The equivalence of the first two is \cite[Prop. 4.3]{ORT}. The equivalence of the last two, in the presence of the stable rank 1 property, is \cite{CEI}.
\end{proof}

By $A\tensor\mathcal H$ we mean Kasparov's standard Hilbert module over $A$, see  \cite[pgs. 34-5]{lance}.
\begin{theorem}Let $A$ be a C*-algebra, and let $a$ and $b$ be positive elements of $A\tensor\compact.$ If $A$ is of stable rank 1, then if $a$ and $b$ are are Cuntz equivalent, then if $a$ is in the Pedersen ideal of $A\tensor\compact$, so is $b$. Without stable rank 1, we have that if $\overline{a(A\tensor \mathcal H)}$ and $\overline{b(A\tensor\mathcal H)}$ are isomorphic as Hilbert $A$-modules, then if $a$ is in the Pedersen ideal of $A\tensor\compact$, so is $b$.  
\end{theorem}
\begin{proof} In view of lemma \ref{lem:isos}, it is sufficient to prove that if  $a\blackadar b,$ and if $a$ is in the Pedersen ideal of $A\tensor\compact$,  then $b$ is in the Pedersen ideal.
Now, the equivalence relation $\blackadar$ is generated by two equivalence relations, which we consider one at a time. First, we have the equivalence relation given by the condition that two positive elements generate the same hereditary subalgebra. Suppose that, as above, one of them is in fact in the Pedersen ideal $K(A\tensor\compact).$ But proposition \ref{prop:PedersenAlgebras} shows that then the entire hereditary subalgebra is in the Pedersen ideal $K(A\tensor\compact),$ so then  a generator of the hereditary subalgebra is in the Pedersen ideal.  
Secondly, consider the equivalence relation where  positive elements $a$ and $b$ are equivalent in a C*-algebra $A$ if there is an element $x\in A$ such that $a=x^*x$ and $b=xx^*.$ But then, for example, by proposition 5.6.2 in \cite{Pedersen}, if $a$ in the Pedersen ideal, then so is $b.$ 
 This completes the proof that if  $a\blackadar b,$ and $a$ is in the Pedersen ideal of $A\tensor\compact$,  then $b$ is in the Pedersen ideal.
\end{proof}

\begin{remark} It follows that in the stable rank 1 case, elements of a C*-algebra which are in the class of a projection in the Cuntz semigroup are contained in the Pedersen ideal. Generally there are also purely positive elements in the Pedersen ideal, as can be seen from  proposition \ref{prop:compactly.generated}.  \end{remark}
The next result would be evident if $c_1$ and $c_2$ were simple tensors, and the bidual were of finite type. Hence, this result suggests that  open projections of a tensor product behave somewhat like simple tensors.

\begin{corollary}Let $C:=\AbA$ be a separable and stably finite C*-tensor product, with almost unperforated Cuntz semigroup.
Let $c_1$ and $c_2$ be positive elements of $C.$ Then the open projection associated with $c_1$ is Murray-von Neumann subequivalent to the open projection associated with $c_2$ if and only if $\lim_{n\arrow\infty} (\tau_1 \tensor \tau_2)(c^{1/n}_1)  \leq 
              \lim_{n\arrow\infty} (\tau_1 \tensor \tau_2)(c^{1/n}_2)  $
							for each $\tau_1\in LTC(A_1)$ and $\tau_2\in LTC(A_2).$
 \end{corollary}
\begin{proof} Combine Theorem \ref{th:LTC} above with \cite[Th. 5.8]{ORT}.\end{proof}

A relation with the Cuntz semigroup becomes more evident if we specialize the situation somewhat.
\begin{corollary}Let $C:=\AbA$ be a separable, simple, and stably finite C*-tensor product. Assume also that the Cuntz semigroup is almost unperforated.
Let $c_1$ and $c_2$ be positive elements of $C$ whose open projections are not contained in $C.$ Then  $c_1$ is Cuntz subequivalent to $c_2$ if and only if $\lim_{n\arrow\infty} (\tau_1 \tensor \tau_2)(c^{1/n}_1)  \leq 
              \lim_{n\arrow\infty} (\tau_1 \tensor \tau_2)(c^{1/n}_2)  $
							for each $\tau_1\in LTC(A_1)$ and $\tau_2\in LTC(A_2).$
\label{cor:ORT2} \end{corollary}
\begin{proof}By remark $3''$ in section 5 of \cite{ORT}, with these hypotheses Murray von Neumann subequivalence of open projections implies Cuntz subequivalence of the elements $c_i.$ Thus this corollary follows from the more general one above. \end{proof}

\subsection{Tensor products of Cuntz-Pedersen spaces} 
Recall that two positive elements $x$ and $y$ are said to be equivalent in the Cuntz-Pedersen sense if $x=\sum a_{i}^{*} a_i $ and $y=\sum a_{i} a_i^{*}, $ where the sums are   norm-convergent or finite. Let us denote this form of equivalence by $\cp.$ 
 Cuntz and Pedersen defined a real Banach space $A^q$ that is the quotient of $A$ by the subspace $A_0$ of elements of the form $x-y,$ where $x\cp y.$ 

The finite-dimensional case of the following is \cite[Proposition 6.12]{CP}. The norm on $A^q\betatensor B^q$ is defined in the last two lines of the proof. 
\begin{corollary}Let $A$ and $B$ be C*-algebras. Then $(\AbB)^q = A^q\betatensor B^q.$ 
\label{cor:CP1}  \end{corollary}
\begin{proof} Recall that $\algtensor$ denotes the algebraic tensor product. Lemma 6.11 in \cite{CP} implies that 
$$ (A\algtensor B) \cap (\AbB)_0 = A_0 \algtensor B_0,$$
where the subscripts denote the abovementioned subspace of elements of the form $x-y,$ where $x\cp y.$  
There is thus an injective map 
$$\imath\colon \frac{A}{A_0 } \algtensor \frac{B}{B_0 } \arrow \frac{\AbB}{(\AbB)_0}.$$ 
Let $T$ now denote the real linear linear space of self-adjoint tracial functionals on a C*-algebra. If the image of the map $\imath$ is not weakly dense, then by Hahn-Banach some nonzero element $\tau$ of $T(\AbB)$ is zero on the image 
$\imath{\left( \frac{A}{A_0} \algtensor \frac{B}{B_0}\right)}.$ On the other hand, elements of the convex hull of $T(A)\tensor T(B)$ are zero on the image of $\imath$ but have been shown to be weakly dense in $T(\AbB).$ Thus, the Hahn-Banach theorem provides a contradiction unless the image of the map $\imath$ is weakly dense. The image is a convex set, so it is norm-dense with respect to the norm $\beta.$ Since the map $\imath$ is injective, we define the norm $\beta$ on $\frac{A}{A_0 } \algtensor \frac{B}{B_0 }=A^q \algtensor B^q$ by $x\mapsto \| \imath(x)\|_\beta.$ 
Taking the completion of $\frac{A}{A_0 } \algtensor \frac{B}{B_0 }$ with respect to this norm, we have  $ A^q\betatensor B^q=(\AbB)^q.$ 
 \end{proof}

Cuntz and Pedersen showed \cite[th.7.5]{CP} that for simple C*-algebras, two positive elements are Cuntz-Pedersen equivalent if and only if they are equal under all \lsc traces. It turns out that simplicity is not necessary \cite{Robert2009}. We thus have the Corollary: 

\begin{corollary} Consider a C*-algebraic tensor product $\AbB.$ 
The following are equivalent, for positive elements:
\begin{enumerate}\item $c\tensor d \cp a\tensor b,$ 
\item $\lambda c\cp a$ and $d\cp \lambda b$ for some positive scalar $\lambda.$ 
\end{enumerate}\label{cor:CP2}
\end{corollary}
\begin{proof}By the above remarks, we may as well redefine $\cp$ to mean equality under all \lsc traces. Suppose that (2) holds, so that $\lambda c\cp a$ and $d\cp \lambda b.$ Let $\tau_1$ and $\tau_2$ be \lsc traces of $A$ and $B$ respectively, taking values in $[0,\infty].$ By a routine argument with monotone sequences, $\tau_1(a)= \lambda \tau_1(c)$ and $\lambda \tau_2(b)= \tau_2(d).$ By Theorem \ref{th:LTC} and a convexity argument, it follows that $\tau(a\tensor c)=\tau(b\tensor d)$ for all \lsc traces $\tau$ on $\AbB.$ 

Now, suppose that (1) holds. Thus, $c\tensor d$ and $ a\tensor b$ are equal under traces.
There are two cases: either these are zero under all traces, or they are not. If not, then, applying the slice map $\tau\tensor\Id$ coming from a \lsc trace $\tau$ of $A,$ we note that $\tau(c)d$ and $\tau(a)b$ are equal under all \lsc traces of $B.$  Then   $\lambda d\cp  b$ for some positive scalar $\lambda.$ 
Similarly, $\lambda' c\cp a$ for some scalar $\lambda',$ and we can use the fact that  $c\tensor d$ and $ a\tensor b$ are equal under traces to show that $\lambda\cdot \lambda'=1.$ Thus, (2) holds. If, on the other hand,  $c\tensor d$ and $ a\tensor b$ were zero under all traces, then these elements are in fact zero (for example, by \cite[Th.4.8]{CP}) and again (2) holds.
\end{proof}

\section{Towards a K\" unneth sequence for Cuntz semigroups}

We begin with a discussion.
 
In the simple and stably finite case, it can be shown that the purely positive elements coincide with the purely non-compact elements. The purely non-compact  elements are defined \cite{ERS} to be those elements which, if they are compact in some quotient, are in fact properly infinite in that quotient. The purely non-compact elements have the useful technical property that they form a subsemigroup of the Cuntz semigroup, and in this respect they can be easier to handle, in the nonsimple case, than the purely positive elements. 

In the simple and exact case, we obtain a large amount of information about the purely noncompact elements by combining our tensor product result with the fact that the dual $F(\Cu(\cdot))$ of the (subsemigroup of) purely noncompact elements of the Cuntz semigroup $\Cu(\cdot)$ consists of dimension functions coming from \lsc traces. 
 
 In the simple, exact, and stably finite case, all dimension functions come from densely finite \lsc traces. 
The Cuntz semigroup $\Cu(A)$ can be defined to be an ordered semigroup defined by the set of open projections of $A\tensor\compact,$ modulo an equivalence relation on open projections \cite{ORT}. In this picture of the Cuntz semigroup, the dimension functions are given by the natural pairing between open projections and traces extended to weights on the double dual. 

  Now, recall that for compact manifolds, $X$ and $Y$, one could use partitions of unity to prove that a continuous positive function on $X\times Y$ can be approximated pointwise from below by a finite sum of positive functions $\sum f_i (x)g_i (y).$ Dini's theorem then implies uniform approximation. We prove a similar result in the setting of additive and homogeneous functions on cones, using the C*-algebraic structure at hand. This generalizes the main technical lemma of \cite{arxiv}. Since we are in a not necessarily nuclear situation we   give a detailed proof.
	
	\begin{lemma} Let $h\in \AbB$ be a positive element of the Pedersen ideal of the tensor product C*-algebra $\AbB.$  Then, the additive and homogeneous function $\widehat{h}$ on  $LTC(A)\times LTC(B)$ given by evaluating the open projection associated with $h$ on $LTC(A)\times LTC(B)$ is continuous and is approximated pointwise from below by sums of additive, homogeneous, non-negative, and continuous functions $\sum f_i (x)g_i (y),$ where $x$ is in $LTC(A)$ and $y$ is in $LTC(B).$ \label{lem:pou}\end{lemma}
\begin{proof} 
Corollary \ref{cor:restriction} implies the well-known fact that the function $\widehat{ h }$ on dimension states given by pairing the open projection $\lim h^{1/n}$ with \lsc traces extended from  $\AbB$  is pointwise finite. Continuity follows from \cite[Prop. 5.3]{ERS}.   By Corollary  \ref{cor:ped.approx}  we can approximate $h^{1/n}$ from below by finite sums of elementary tensors of positive elements, of the form $\sum a_i\tensor b_i.$ When such a sum is evaluated on an extremal trace $\tau$ of $\AbB,$ by Theorem \ref{th:TracialStates}, the trace $\tau$ factorizes as  $\tau_a\tensor \tau_b,$ and thus $\tau(\sum a_i\tensor b_i)=\sum \tau_a(a_i)\tensor \tau_b(b_i).$ Taking pointwise suprema we approximate $\widehat{h}.$ The case of non-extremal traces follows by a convexity argument. 
		\end{proof}

	An abelian semigroup endowed with a scalar multiplication by strictly positive real numbers is termed a non-cancellative cone in \cite{ERS}. A real-valued function on a partially ordered set is called \lsc if it preserves directed suprema. Thus, we have a class of \lsc functions on non-cancellative cones. Denoting by $T(A)$ the cone of all, possibly not densely finite, traces on a C*-algebra $A,$ we may consider the additive and homogeneous \lsc functions on $T(A).$ Denote by $L(T(A))$ those \lsc functions that are suprema of sequences $(h_n)$ where $h_n$ is continuous at each point where $h_{n+1}$ is finite.  For the cone of densely defined \lsc traces, the situation is technically simplified, as we have seen, by the fact that these traces belong to the topological vector space of continuous linear functionals on the Pedersen ideal.

We recall the well-known fact that in the $\Z$-stable case, the purely non-compact part of the Cuntz semigroup is determined by its dual \cite{Robert2013}. Specializing slightly, this implies: 
  \begin{proposition} Let $A$ be a simple, exact, stably finite, and $\Z$-stable C*-algebra. The ordered semigroup of purely positive elements of $\Cu(A)$ is isomorphic to the ordered semigroup $L(LTC(A)).$  \label{prop:Zstable}\end{proposition}
	\begin{proof} Since $A$ is $\Z$-stable, its Cuntz semigroup is unperforated. Then Theorem 6.6 in \cite{ERS} states that the ordered semigroup of purely non-compact elements of the Cuntz semigroup is isomorphic to $L(F(\Cu(A))).$ Here, $F$ denotes the cone of dimension functions on the Cuntz semigroup and $L$ is a dual space of lower-semi-continuous functions.  In the simple case the purely non-compact elements are either infinite projections or purely positive elements \cite[Prop. 6.4.iv]{ERS}. But we have assumed that there are no infinite projections. 

	Blanchard and Kirchberg \cite[pg.486]{BlanchardKirchberg}, see also  \cite[Prop. 4.2]{ERS},  show that there is a one-to-one correspondence between \lsc dimension functions and locally lower semi-continuous local quasi-traces, but as they point out,  results of Haagerup's \cite{Haagerup2014} imply that the quasi-traces are traces when the algebra is exact.  Thus, there is in the exact, simple, and stably finite case a structure-preserving one-to-one correspondence between \lsc traces in $LSC(A)$ and dimension functions in $F(\Cu(A)).$  We conclude that the semisubgroup of purely positive elements is isomorphic to $L(LTC(A)),$ where $LTC(A)$ denotes the cone of densely finite \lsc traces on the C*-algebra $A,$ and $L$ is as defined above. 
	\end{proof} 
	We may drop simplicity in the above if we strengthen the stable finiteness condition:
	\begin{corollary} Let $A$ be an exact, residually stably finite, and $\Z$-stable C*-algebra. The ordered semigroup of purely non-compact elements of $\Cu(A)$ is isomorphic to the ordered semigroup $L(LTC(A)).$  \label{prop:nonsimpleZstable}\end{corollary}
	
We now consider the map $\Cu(A)\times\Cu(B)\arrow \Cu(\AbB)$ defined, at the level of positive operators, by $(a,b)\mapsto a\tensor b.$ It is routine to check that this map is compatible with the equivalence relation, \textit{i.e.} that $a\tensor b$ is equivalent to $a'\tensor b$ if $a$ is equivalent to $a'.$ This map is in principle well-understood for projection-class elements, see Appendix B for an exposition. Thus, we are interested in the case of purely non-compact elements. We will show that the map is sometimes surjective onto such elements. Injectivity cannot generally be expected.

The above bilinear map factors through a map of tensor products, \cite[sect. 6.4]{APT14}, because we have a  natural Cuntz semigroup morphism from $\Cu(A)\tensor \Cu(B)$ to $\Cu(A\tensor_{max} B),$ and then  the tensor quotient map  takes $\Cu(A\tensor_{max} B)$ onto $\Cu(A\betatensor B).$ The range is a sub-object of the Cuntz semigroup, \textit{i.e.} is closed under sequential suprema,  and we may ask when the purely non compact elements are contained in the range. 

The following corollary computes the purely positive part of the Cuntz semigroup of a tensor product in many cases, and in \cite{arxiv} we used an early similar result as a definition of a tensor product of Cuntz semigroups. Since the K\" unneth sequence in principle computes the projection class part of the Cuntz semigroup, at least in many cases (see Appendix B), we thus have determined the Cuntz semigroup of a tensor product.
\begin{corollary}Suppose that $A$ and $B$ are simple, $\Z$-stable, and exact. Then if the C*-tensor product $C:=\AbB$ is separable and stably finite, the  map $\Cu(A)\tensor \Cu(B)\arrow \Cu(\AbB)$is onto the subsemigroup of purely positive elements of $\Cu(\AbB).$\label{cor:useAPT}\end{corollary}	
	\begin{proof} Proposition  \ref{prop:Zstable} shows that the purely positive elements of $\AbB$ are given by $L(LTC(\AbB)),$ and that the purely positive elements of $A$ and of $B$ are given by $L(LTC(A)$ and $L(LTC(B)),$ respectively. On the other hand, by Theorem \ref{th:LTC}, a densely defined extremal \lsc trace $\tau$ on $C$ is a tensor product $\tau_1\tensor \tau_2$ of traces on $A$ and on $B,$ respectively. Thus, we consider the map 
	$$L(LTC(A))\times L(LTC(A))\arrow L(LTC(A)\times LTC(A)).$$ 
	In the separable case, Lemma \ref{lem:pou} provides a sequence of continuous functions of the form $\sum f_i (x)g_i (y)$ that approximates (pointwise, from below) a continuous function in $L(LTC(A)\times LTC(A)).$ On the other hand, in the simple and stably finite case, any element of $L(LTC(A)\times LTC(A))$ can be obtained as a supremum of continuous functions. It then follows that the range contains the purely positive elements. The universal property of tensor products then provides that the map factors through a tensor product, as claimed.
	\end{proof}
	
The above result generalizes a result from \cite{arxiv}. Injectivity of the tensor product map on the whole of the Cuntz semigroup cannot generally be expected, for example because of the fact that the tensor product of a projection and a purely positive element will usually be purely positive. It is quite rare for the tensor product map to induce an isomorphism of Cuntz semigroups. For an example, recall:
	\begin{corollary}[\protect{\cite[Cor. 4.3]{arxiv}}] If an unital C*-algebra $A$ is simple, stably projectionless, stably finite, nuclear, Z-stable, satisfies the UCT, has stable rank one, has $K_1 (A) = {0}$ and $A\tensor A$ has stable rank one then the tensoring map
$t : \Cu(A) \tensor \Cu(A) \rightarrow \Cu(A \tensor A)$
is an isomorphism.
\end{corollary}

	\section{Ideals}
	
	We now consider \lsc traces on two-sided closed ideals of C*-algebras. This is motivated by Blanchard and Kirchberg's introduction of dimension functions that are densely defined on such ideals, see \cite[Prop. 2.24.ix]{BlanchardKirchberg}.	If we assume exactness, then the dimension functions are determined by \lsc traces, rather than quasi-traces. 
 Let us thus define $T_I (A)$ to be the cone of \lsc traces on $A_+$ that are densely defined on a closed two-sided ideal $I,$ and are not necessarily densely defined outside $I.$ We regard them as having the value $\infty$ in the extended real number system rather than  actually being undefined. Traces of this type were studied systematically in \cite{BlanchardKirchberg,ERS}.  In some cases there may be traces which are nontrivial but are not densely defined on an ideal: we have essentially nothing to say about this case. However, the case of traces densely defined on an ideal is somewhat approachable, especially if the ideal lattice is finite.

The key observation is due to Blanchard and Kirchberg, \cite[Prop. 2.16]{BlanchardKirchberg}, which is that in the exact case, every ideal is, up to closure, the linear sum of the elementary ideals $J_1\tensor J_2$ that it contains. We recall, see \cite[pg. 980]{ERS}, that the ideals form a continuous lattice, within $T(A)$, and that the natural topology on $T_I (A)$ is the topology of pointwise convergence on the positive elements of the Pedersen ideal of $I.$ 

 We can extend our main result to this setting by simply replacing, in the proof, the Pedersen ideal of $A_i$ by the Pedersen ideal of a given closed ideal $J_i.$

\begin{corollary} Let $C:=\AbA$ be a C*-tensor product of C*-algebras $A_1$ and $A_2.$
Let $V$ be an ideal of $C$ that is of the form $J_1\tensor J_2.$
 Then $\T$ is a homeomorphism from $LTC_{J_1} (A_1)\betatensor LTC_{J_2}$ onto $LTC_V (C),$ where $LTC_V$ denotes the cone of  \lsc traces that are densely defined on $V.$ \label{th:ideals}
\end{corollary}
\begin{proof}Let $k_\lambda$ and $k'_\lambda$ be increasing approximate units (nets) contained in $\Ped(J_1)$ and $\Ped(J_2)$ respectively.  Then $k_\lambda \tensor k'_{\lambda'},$ with the product order on the index set, is  an approximate unit for the ideal $V,$ and we proceed as in the proof of Theorem \ref{th:LTC}.   \end{proof}

\begin{remark}The above theorem gives us a reasonably concrete expression for the traces which are densely defined on an elementary ideal $J_1\tensor J_2.$ Note that the traces are not necessarily infinite outside the given ideal, thus this is a large class of traces in general. We can then intersect $T_{J_1\tensor J_2}$ and $T_{J_3\tensor J_4}$ within $T(C)$ to obtain the traces that are densely defined on both ${J_1\tensor J_2}$ and ${J_3\tensor J_4}.$ Thus we can in principle compute $T(C),$ and hence $\Cu(C),$ in some cases.\end{remark}

	Recall that in  \cite{Skandalis1988} was defined a notion of $K$-nuclearity, weaker than nuclearity, and a characterization was found in terms of bivariant $KK$-theory. The equivalent question for the Cuntz semigroups would be:
	\begin{question} We could say that a $C^*$-algebra $A$ is $\Cu$-nuclear if the map $\imath\colon A \tensor_{\mbox{max}} B \rightarrow A \tensor_{\mbox{min}} B$ induces an isomorphism of Cuntz semigroups for all C*-algebras $B$. Are such C*-algebras necessarily nuclear?
	\end{question}
	
	Also, it is possible that it is an interesting problem to fully determine the Pedersen ideals of not necessarily elementary closed ideals of tensor products in terms of the tensor factors, because this would lead to a better understanding of traces that are densely defined on such ideals. Even partial information, as we have seen, can be useful.
	
\section{Appendix A: a topological vector space approach to tensor products of Cuntz semigroups}

In the early paper \cite{arxiv}, we defined a tensor product of Cuntz semigroups by first representing them as affine functions on the tracial simplexes and then declaring that the tensor product is the set of bi-affine functions on the two simplexes. At the level of simplexes this is known as the projective tensor product construction, see  Semadeni \cite[p. 304]{Semadeni} and also \cite[p.475]{NamiokaPhelps}. In \cite{arxiv} we computed the resulting tensor product of Cuntz semigroups in a few cases. But might there not be other tensor products that could be defined, much as in the case of topological vector spaces? 

Of course, the Cuntz semigroups are not topological vector spaces. However,  let us attempt to take guidance from the theory of topological vector space tensor products. The first question that arises is if there exists any natural topology on the elements of a Cuntz semigroup.
The natural topology on the projection-class elements of the Cuntz semigroup is the discrete topology. However,  when considering  purely positive elements, the set of all dimension functions is much like a set of seminorms defining a topological vector space. We attempt to press the analogy further.

Recall that there exists a very general definition  of the so-called inductive (or injective) tensor product of topological vector spaces, due to   Gro\-then\-dieck, see \cite[d\'efinition 3]{grot}. Generalizing this definition to the setting of Cuntz semigroups, we form the algebraic tensor product of abelian semigroups, see \cite{grillet}, and then view  elements of $Cu(A)\otimes_{alg} Cu(B)$ as   functions on $D(A)\times D(B),$ where $D(A)$ and $D(B)$ denote the dimension functions on the Cuntz semigroup(s). The inductive topology is the (initial) topology induced by this embedding. 
We then take the topological completion of $Cu(A)\otimes_{alg} Cu(B)$ with respect to this topology (see \cite[ex. 6.L, and pp. 195-6]{kelley} for information on completions). We only need to perform the above construction on the set of those elements whose image under the tensor product map is purely positive. For brevity we refer to these elements as the purely positive elements. We have that for the purely positive elements of
 $Cu(A)\otimes_{alg} Cu(B)$, an increasing sequence $x_n$ converges (pointwise) to an element of the completion if and only if $(d_1, d_2)(x_n)$ converges for all $d_1\in D(A)$ and $d_2 \in D(B).$ The limit of the sequence exists in the completion, and we define the inductive tensor product of Cuntz semigroups, denoted $Cu(A) \otimes Cu(B),$ to be $Cu(A)\otimes_{alg} Cu(B)$ augumented by the set of all such limits.  In the nonunital setting, we would have to allow the extended real number system.

We also define the minimal embedding tensor product, given as follows. Consider the natural tensor product map $t\colon Cu(A)\otimes_{alg} Cu(B)\rightarrow Cu(A\otimes_{min} B).$  This map induces a uniformizable topology on its domain (sometimes called the inital topology). Taking, then, the  completion of the domain with respect to this topology, and proceeding as in the previous paragraph, we obtain the minimal embedding tensor product, $Cu(A)\otimes_{min} Cu(B).$ If, in the above, we replace   $Cu(A\otimes_{min} B)$ by $Cu(A\otimes_{max} B),$ then we obtain the maximal embedding tensor product, denoted $Cu(A)\otimes_{max} Cu(B).$ 

Comparing the tensor product of \cite{APT14} with the above tensor products seems to be an interesting but lengthy project, since the definition there is based on a completely different approach.

In the presence of a separability condition, we can introduce a Hausdorff-type pseudometric
$d(x,y):=\sum 2^{-n} d_n(x,y)$ where $x$ and $y$ are purely positive elements and $d_n$ is some countable dense set of dimension functions. A different approach, based on unitary orbits in the algebra, also gives a  pseudometric in the special case considered in \cite{CiupercaElliott}.

\section{Appendix B: dimension monoids of tensor products}

 In this mostly expository appendix,  we consider the projection class elements of the Cuntz semigroup under the tensor product map.  For algebras of stable rank one, the Cuntz class of a positive element is given by a projection if and only if $0$ is not in the spectrum or if it is an isolated point of the spectrum.

If $a \in M_{\infty}(A)^+, b \in M_{\infty}(B)^+$ are positive elements then it follows that $a \otimes b$ is a positive element in $M_{\infty}(A \otimes B)^+$. This induces a bilinear morphism from $Cu(A) \times Cu(B)$ to $Cu(A \otimes B)$ which in turn induces a natural Cuntz semigroup map 
$$ t: Cu(A) \otimes_{alg} Cu(B)  \rightarrow Cu(A \otimes B)$$
$$t([a] \otimes [b])=[a \otimes b].$$
 
We expect that in favorable cases, the Universal Coefficient Theorem can be used to study the above map:

\begin{lemma}\label{lem:monoids}
If a C*-algebra $A$ is simple, separable, unital, nuclear, ${\mathcal Z}$-stable, stably finite, has finitely generated $K_0(A)$, $K_1(A)=\{0\},$ and satisfies the UCT, then the natural map $$ Cu(A) \otimes_{alg} Cu(A)  \rightarrow Cu(A \otimes A),$$ given by
 $$([a] \otimes [b])\mapsto[a \otimes b]$$ is an isomorphism from $V(A) \otimes _{alg} V(A)$ to $V (A \otimes A)$, \textit{i.e.}
 $$V(A) \otimes_{alg} V(A) \cong V(A \otimes A).$$
\end{lemma}
\begin{proof}
We remark that the hypotheses on $A$ imply stable rank 1 by \cite{rorSR1}. Since the tensor product, $A\otimes A,$ also is simple, separable, unital, nuclear, and ${\mathcal Z}$-stable, it follows that the tensor product, $A\otimes A,$ also has stable rank one.
The K-theory group  $K_0(A)$ is determined by the subsemigroup $V(A)$ of the Cuntz semigroup, in the sense that $K_0(A)$ is the Grothendieck group generated by $V(A)$:
$$ K_0(A)= G(V(A)).$$
  Our algebra $A$ is assumed to satisfy the UCT, so then $A$ will satisfy the K\"unneth formula for tensor products in $K$-theory \cite{Schochet1982}, see also the partial counterexample due to Elliott in \cite{SchochetRosenberg} which means that we must assume finitely generated $K_0$-group:

  \begin{multline*}0 \rightarrow K_0(A) \otimes K_0(A) \oplus K_1(A) \otimes K_1(A) \rightarrow K_0(A \otimes A) \rightarrow\\ Tor(K_0(A),K_1(A)) \oplus Tor(K_1(A),K_0(A)) \rightarrow 0.\end{multline*}

 It follows from the above exact sequence that $$K_0(A) \otimes K_0(A) \rightarrow K_0(A \otimes A)$$ is an injective map. Since $K_1(A)=\{0\}$ it follows that we have an order-preserving isomorphism $$t:K_0(A) \otimes K_0(A) \rightarrow K_0(A \otimes A).$$ We will show that when restricted to the positive cones, this isomorphism becomes equivalent to the given map.
        Since $A$ has stable rank 1, there is an injective map $i\colon V(A)\rightarrow K_0 (A)$ and $V(A)$ has the cancellation property. Taking algebraic tensor products of semigroups,  we consider $V(A)\otimes_{alg} V(A).$  Taking the (semigroup) tensor product of maps,  we obtain a map $ i\otimes_{alg}i\colon    V(A)\otimes_{alg} V(A) \rightarrow  K_0(A) \otimes_{alg} K_0(A)$ where  $K_0(A) \otimes_{alg} K_0(A)$ is a semigroup tensor product of abelian groups. 
         Moreover, the semigroup tensor product of abelian groups coincides with the usual tensor product of abelian groups (by Proposition 1.4 in \cite{grillet} and the remarks after that Proposition).  We note 
that the map $ i\otimes_{alg}i$ is an injective map (using Lemma \ref{lem:tech} below).   Since $A$ is stably finite and the positive cone of a tensor product of finitely generated ordered abelian groups is the tensor product of the positive cones of the ordered abelian groups, it follows that the range of the map $ i\otimes_{alg}i$ is exactly the positive cone of  $K_0(A) \otimes K_0(A).$

        Composing with the above injective map $t$, we obtain an injective map from $V(A)\otimes_{alg} V(A)$ to $K_0(A \otimes A),$ which takes an element $p \otimes_{alg} q$ to $p \otimes q.$ Since $t$ is an order isomorphism, the range of $t\circ (i\otimes_{alg} i)$ is exactly the positive cone of $K_0 (A\otimes A).$ We now observe that the map we have obtained  is in fact equal to the given map, because, as   $A \otimes A$ has stable rank 1 and is stably finite, it follows that $V(A \otimes A)$ is embedded in $K_0(A \otimes A)$ as the positive cone. It then follows that $t\circ (i\otimes_{alg} i),$ which acts on elements by taking $p\otimes_{alg} q$ to $p\otimes q,$ is in fact an injection of $V(A)\otimes_{alg} V(A)$ onto $V(A\otimes A).$ Evidently, this map coincides with the given map.

\end{proof}

The technical lemma on semigroups referred to above  is Proposition 17 in \cite{fulp}.
\begin{lemma} Let $V$ be a semigroup, and let $G$ denote the formation of the enveloping group. We have $$G(V \otimes V) = G(V)\otimes_Z G(V),$$ where $\otimes$ denotes the tensor product of semigroups, and $\otimes_Z$ denotes the tensor product of abelian groups.
\label{lem:tech}\end{lemma}


\begin{thebibliography}{99}

\bibitem{AkemannJohnson}   \textsc{Akemann, C.A;  Johnson, B.E. } {Derivations of nonseparable C*-algebras,} \textit{J. Funct. Anal.} \textbf{33} (1979), (3), 311--331 \mrev{0549117} \zbl{ 0423.46049}
\bibitem{Ando2004} \textsc{Ando, T. }
 {Cones and norms in the tensor product of matrix spaces,}
\textit{Linear Algebra Appl.} \textbf{379} (2004), 3--41 
                  \mrev{2039295 (2005e:46030)} \zbl{ 1041.15021}
\bibitem{APT14} \textsc{Antoine, R.;  Perera, F;  Thiel, H.}  {Tensor products and regularity properties of Cuntz semigroups},\textit{ Mem. AMS} \textbf{251} (2018) 
            \mrev{3756921} \zbl{07000078}
\bibitem{Asimov1969} \textsc{Asimow, L.}  {Extremal structure of well-capped convex sets}, \textit{Trans. Amer. Math. Soc.} \textbf{138} (1969) 363--375
     \mrev{0240586 (39 \#1933)} \zbl{ 0175.13101}
\bibitem{Asimov1968}  \textsc{Asimow, L. } {Universally well-capped cones}, \textit{Pacific J. Math.} \textbf{26} (1968) 421--431.
            \mrev{0234250 (38 \#2568) } \zbl{ 0167.42102}
\bibitem{Arch1975}   \textsc{Archbold, R.J.}  {On the centre of a tensor product of C*-algebras}, J. \textit{London Math. Soc.} (2) \textbf{10} (1975), 257--262
      \mrev{0402512 (53 \#6331)} \zbl{ 0303.46057}
\bibitem{Arch1977}  \textsc{Archbold, R.J.}  {An averaging process for $C^*$-algebras related to weighted shifts}, \textit{Proc. London Math. Soc.} (3) \textbf{35} (1977), 541--554
\mrev{0512359 (58 \#23620)} \zbl{
0373.46070}
\bibitem{Arch1978}  \textsc{Archbold, R. J.}  {On the norm of an inner derivation of a $C^*$-algebra}, \textit{Math. Proc. Camb. Phil. Soc.} \textbf{84} (1978), 273--291
   \mrev{0482236 (58 \#2316)} \zbl{ 0388.46038}
\bibitem{blackadar88} \textsc{Blackadar, B. }
{Comparison theory for simple C*-algebras. Operator algebras and applications}, pp. 21--54,
{London Math. Soc. Lecture Note Ser.,} \textbf{135}, \textit{Cambridge Univ. Press, Cambridge, }1988
   \mrev{0996438 (90g:46078)} \zbl{ 0706.46043}
\bibitem{BlanchardKirchberg} \textsc{Blanchard, \'E.;  Kirchberg,  E.}  {Non-simple purely infinite C*-algebras: the Hausdorff case}. \textit{J. Funct. Anal. }\textbf{207} (2004), 461--513
      \mrev{2032998 (2005b:46136)} \zbl{ 1048.46049}
\bibitem{Choquet} \textsc{Choquet, G.}   {Repr\'esentation Int\'egrale,}\textit{ Measure Theory and its Applications}, p. 114--125, ed. Belley, J.M., Dubois, J., Morales, P., (1983)  \textit{Springer, New York}
      \mrev{0729530 (85c:46007)} \zbl{ 0551.46006}
\bibitem{CiupercaElliott} \textsc{Ciuperca, A.; Elliott, G. A.}    {A remark on invariants for C*-algebras of stable
rank one,} \textit{Int. Math. Res. Not. IMRN} (2008),  \textbf{5}
       \mrev{2418289 (2010f:46094)}        \zbl{1159.46036} 
\bibitem{CEI} \textsc{Coward, K.T.; Elliott, G. A.;   Ivanescu, C.} 
 {The Cuntz semigroup as an invariant for C*-algebras} 
\textit{J. Reine Angew. Math.} \textbf{623} (2008), pp. 161---193
 \mrev{2458043 (2010m:46101)} \zbl{ 1161.46029}
\bibitem{Cohen1959} \textsc{Cohen, P. J. }  { Factorization in Group Algebras,} \textit{Duke Math. J., }\textbf{26} (1959) 199--205
\mrev{0104982 (21 \#3729)} \zbl{ 0085.10201} 
\bibitem{Combes1969} \textsc{Combes, F.}  {Sur les faces d'une C*-alg\`ebre}, \textit{Bull. Sci. Math.,}  \textbf{93} (1969), 37--62
           \mrev{0265947 (42 \#856)}   \zbl{0177.17801}
\bibitem{CP} \textsc{Cuntz, J.; Pedersen, G.K.}   {Equivalence and traces on C*-algebras},\textit{ J. Funct. Anal.} \textbf{33} (1979)  135--164
\mrev{0546503 (80m:46053)} \zbl{ 0427.46042}
\bibitem{Davies1969} \textsc{Davies, E.B.}  {Decomposition of traces on separable C*-algebras,} \textit{Quart. J. Math., Oxford, series 2,}  \textbf{20}, (1969), p.97--111
    \mrev{0240638 (39 \#1984)} \zbl{ 0172.17903}
\bibitem{Davies1968} \textsc{Davies, E.B. }  {On the Borel structure of C*-algebras},\textit{ Comm. math. Phys.} \textbf{8}, 147--163 (1968)
         \mrev{0231209 (37 \#6764)} \zbl{ 0153.44701}
\bibitem{Dixmier1949} \textsc{Dixmier, J. }  {Les anneaux d'op{\' e}rateurs de classe finie,} \textit{Ann. Sci. {\' E}cole Norm. Sup.} (3) 66 (1949) 209--261
       \mrev{0032940 (11,370c)} \zbl{0036.35802 }
\bibitem{Dixmier} \textsc{Dixmier, J. }{Les C*-alg\`ebres et leurs repr\'esentations} \textit{Gauthier-Villars, Paris} (1964)
\mrev{0171173 (30 \#1404)} \zbl{ 0152.32902}
\bibitem{DNR} \textsc{D\u{a}sc\u{a}lescu, S.;  N\u{a}st\u{a}sescu, C.; Raianu, \c{S}.}{Hopf Algebras,} \textit{Marcel Dekker Inc., New York}, 2001
       \mrev{1786197 (2001j:16056)} \zbl{0962.16026 }
\bibitem{Engelking} \textsc{Engelking, R.} {General Topology}, \textit{Heldermann--Verlag, Berlin} (1989)
        \mrev{1039321
(91c:54001)} \zbl{ 0684.54001}
\bibitem{Effros} \textsc{Effros, E.G. } {Structure in simplexes}, \textit{Act Math., }\textbf{117}  p. 103--121 (1967)
\mrev{0203435 (34 \#3287)} \zbl{ 0154.14201}
\bibitem{EffrosHahn}\textsc{Effros, E.G.;  Hahn, F. } {Locally compact transformation groups and C*-algebras,} \textit{Memoirs of the Amer. Math. Soc.,} No. \textbf{75},  {Amer. Math. Soc. Providence, R.I.} (1967)
\mrev{0227310 (37 \#2895)} \zbl{ 0184.17002}
\bibitem{ERS} \textsc{Elliott, G.A.; Robert, L.; Santiago,  L. }  {The cone of lower semicontinuous traces on a C*-algebra}, \textit{Amer. J. Math.,} \textbf{133} (4) 2011,
pp. 969--1005
\mrev{2823868
(2012f:46120)} \zbl{ 1236.46052}
\bibitem{FHHMZ} \textsc{Fabian, M.; Habala, P.;  H\'ajek, P.;   Montesinos, V.; Zizler, V.} {Banach 
Space Theory,} Springer, New York (2011)
   \mrev{2766381 (2012h:46001)} \zbl{ 1229.46001}
 \bibitem{fulp} \textsc{Fulp, R.}  {Tensor and torsion products of semigroups,}\textit{ Pacific Journal of Mathematics,} \textbf{32} (1970), 685--696
             \mrev{0272928 (42 \#7809)} \zbl{ 0223.20080}
\bibitem{grillet} \textsc{ Grillet, P.A.I.}  {The tensor product of semigroups}, \textit{Trans. Amer. Math. Soc.,} {\bf 138} (1968), 267--281
                \mrev{0237687 (38 \#5968)} \zbl{ 0191.01601}    
\bibitem{grillet2} \textsc{Grillet, P.A.I.}  {The tensor product of cancellative semigroups}, \textit{J. Natur. Sci. and Math. }{\bf 10} (1970), 199–-208
        \mrev{0313427 (47 \#1981)} \zbl{ 0245.20068}
\bibitem{grot} \textsc{Grothendieck,  A.} {Produits tensoriels topologiques et espaces nucl\'eaires}, \textit{Memoirs Amer. Math. Soc.} {\bf 16} (1955), 1--191
                       \mrev{0075539 (17,763c)} \zbl{ 0123.30301}
\bibitem{Guichardet1965} \textsc{Guichardet, A. } {Tensor products of C*-algebras},\textit{ Soviet Math. Dokl}. \textbf{6} (1965),
210--213
\mrev{0176351 (31 \#626)} \zbl{ 0127.07303}
\bibitem{Guichardet1969}  \textsc{Guichardet, A.} {Tensor products of C*-algebras, Part I, Finite Tensor Products}, {Aarhus University Lecture
Notes Series }\textbf{12} (1969)
\zbl{0228.46056} %no \mrev{?
\bibitem{Haagerup2014} \textsc{Haagerup, U. } {Quasitraces on exact C$^*$-algebras are traces}, \textit{C. R. Math. Acad. Sci. Soc. R. Can.,} {\bf 36} (2014), 67--92
            \mrev{3241179} \zbl{ 1325.46055}
\bibitem{Halmos} \textsc{Halmos, P.R.}
{A Hilbert space problem book}, \textit{D. Van Nostrand Co., Inc.,} {1967} 
       \mrev{0208368 (34 \#8178)} \zbl{ 0144.38704}        
\bibitem{arxiv} \textsc{Ivanescu, C.; {Ku\v cerovsk\'y}, D.} {The Cuntz semigroup of the tensor product of C*-algebras}, preprint, arXiv:1501.01038v6
\bibitem{Klee1955} \textsc{Klee, Jr., V.L.}   {Separation properties of convex cones},\textit{ Proc. Amer. Math. Soc. }\textbf{6}
(1955), 313--318
              \mrev{0068113 (16,832a)} \zbl{ 0064.35602}
\bibitem{KadRing} \textsc{Kadison, R. V.;Ringrose, J. R.} {Fundamentals of the theory of operator algebras, Vol. II (Advanced Theory)}, Academic Press, London, 1986
            \mrev{0859186
(88d:46106)} \zbl{ 0601.46054  }    
\bibitem{kelley}  \textsc{Kelley, J. L.} {General Topology}, \textit{Van Nostrand} (1955), reprinted by \textit{Springer}
             \mrev{0370454 (51 \#6681)} \zbl{ 0066.16604}  
\bibitem{KirRor2000} \textsc{Kirchberg,E.;  R\o rdam, M.}  {Non-simple purely infinite C*-algebras,} \textit{Amer. J. Math.} \textbf{122} (2000), pp.  637--666
        \mrev{1759891 (2001k:46088)} \zbl{ 0968.46042}
\bibitem{Lance1973} \textsc{Lance, E. C.}  {On nuclear C*-algebras}. \textit{ J. Funct. Anal.} \textbf{12} (1973), 157--176
            \mrev{0344901 (49 \#9640)} \zbl{ 0252.46065}  
\bibitem{lance} \textsc{Lance, E. C. } {Hilbert $C^{*}$-modules,} \textit{London Math. Soc. 
lecture note series }\textbf{210} (1995)
               \mrev{1325694 (96k:46100)} \zbl{ 0822.46080}
\bibitem{NamiokaPhelps} \textsc{Namioka,  I. ; Phelps, R.R.}  { Tensor products of compact convex sets}. \textit{Pacific J. Math.} \textbf{31} (1969) 469--480
                     \mrev{0271689 (42 \#6572)} \zbl{ 0184.34302}
\bibitem{ORT} \textsc{Ortega, E.;R\o rdam, M.; Thiel, H.}
  {The Cuntz semigroup and comparison of open projections,}
\textit{ J. Funct. Anal.} \textbf{260} (2011), pp. 3474--3493
                    \mrev{2781968
(2012d:46148)} \zbl{ 1222.46043  }
\bibitem{Pedersen1966} \textsc{Pedersen, G.K.}  {Measure theory for C*-algebras}, \textit{Math Scand.,} \textbf{19}, (1966), p. 131--145
             \mrev{0212582 (35 \#3453)} \zbl{ 0177.17701}
\bibitem{Pedersen1968} \textsc{Pedersen, G.K.}  {A decomposition theorem for C*-algebras}, \textit{Math Scand, }\textbf{22}, (1968), p. 266--268
                 \mrev{0253062 (40 \#6277)} \zbl{ 0181.14303}    
\bibitem{Pedersen1968b} \textsc{Pedersen, G.K.}  {Measure theory for C*-algebras} II, \textit{Math Scand.},  \textbf{22}, (1968), p. 63--74
                     \mrev{0246138 (39 \#7444)} \zbl{ 0177.17701}   
\bibitem{Pedersen1969} \textsc{Pedersen, G.K.}  {Measure theory for C*-algebras}, III, \textit{Math Scand.,  }\textbf{25}, (1969), p. 71--93
                     \mrev{0259627 (41 \#4263)} \zbl{ 0189.44504}  
\bibitem{Pedersen1969b} \textsc{Pedersen, G.K.}  {Measure theory for C*-algebras}, IV, \textit{Math Scand., } \textbf{25}, (1969), p. 121--127
                  \mrev{0259627 (41 \#4263)} \zbl{ 0189.44504}
\bibitem{Pedersen1971} \textsc{Pedersen, G.K.}  {C*-integrals: An approach to noncommutative measure theory (Ph.D. thesis)}, Copenhagen (1971)
\bibitem{Pedersen}  \textsc{Pedersen, G.K.} {C*-algebras and their automorphism groups,} \textit{Academic Press, London }(1979)
        \mrev{0548006 (81e:46037)} \zbl{ 0416.46043}
\bibitem{Pedersen1998} \textsc{Pedersen, G.K.}  {Factorization in C*-algebras}, \textit{Expo. Math. }\textbf{16} (1998), p. 145--156
			\mrev{1630695 (99f:46086)} \zbl{ 0912.46054}
\bibitem{PedPet} \textsc{Pedersen, G.K.;  Petersen, N.H.}  {Ideals in a C*-algebra}, \textit{Math Scand} \textbf{27} (1970), p. 193--204
                \mrev{0308797 (46 \#7911)} \zbl{ 0218.46056 }        
\bibitem{Perdrizet1970} \textsc{Perdrizet, F.}  {Topologie et traces sur les C*-alg\`ebres}, \textit{Bulletin de la S.M.F.}, t. \textbf{99}, (1971), p. 193--239
                         \mrev{0293409 (45 \#2486)} \zbl{ 0226.46061}
\bibitem{Robert2009} \textsc{Robert, L.}  {On the comparison of positive elements}\textit{, Indiana U. Math. J.}, \textbf{58} (2009), p. 2509--2515
                              \mrev{2603757
(2011a:46079)} \zbl{ 1191.46047} 
\bibitem{Robert2013} \textsc{Robert, L.} 
 {The cone of functionals on the Cuntz semigroup.} 
\textit{Math. Scand. }\textbf{113} (2013),  161--186
                    \mrev{3145179} \zbl{ 1286.46061}
\bibitem{rorSR1} \textsc{R\o rdam, M.}  {The stable and the real rank of $\mathcal{Z}$-absorbing C*-algebras,} \textit{International J. Math.,} \textbf{15}, No. 10 (2004), 1065--1084
                           \mrev{2106263 (2005k:46164)} \zbl{
1077.46054  }
\bibitem{Schaefer} \textsc{Schaefer, H.H.} {Topological Vector Spaces}, \textit{Macmillan} (1966)
                   \mrev{0193469 (33
\#1689)} \zbl{ 0141.30503}
\bibitem{Semadeni} \textsc{Semadeni, Z.}   {Categorical methods in convexity,} 1967  \textit{Colloquium on
Convexity} (Copenhagen, 1965) pp. 281--307 Proc. Kobenhavns Univ. Mat. Inst.,
Copenhagen 
                   \mrev{0216366 (35 \#7200)} \zbl{ 0146.36205}
\bibitem{Skandalis1988} \textsc{Skandalis, G.}  {Une notion de nucl\'earit\'e en {$K$}-th\'eorie}, \textit{$K$-Theory,} \textbf{1}, (1988), p. 549--573
             \mrev{0953916 (90b:46131)} \zbl{ 0653.46065}
\bibitem{SchochetRosenberg} \textsc{Rosenberg, J.; Schochet, C.}  {The K\"unneth theorem and the universal coefficient theorem for Kasparov's generalized K-functor,} \textit{Duke Math. J. }\textbf{55} (1987), 431--474
          \mrev{0894590 (88i:46091)} \zbl{ 0644.46051}
\bibitem{Schochet1982} \textsc{Schochet, C.} {Topological methods for C*-algebras, II: Geometry resolutions and the
K\"unneth formula,} {\textit{Pacific J. Math.}} \textbf{98} (1982), 443--458
                    \mrev{0650021
(84g:46105b)} \zbl{ 0439.46043}
\bibitem{TTk} \textsc{Tikuisis, A.; Toms, A.S.}  {On the structure of the Cuntz semigroup in (possibly) nonunital C*-algebras}, \textit{Canad. Math. Bull.} \textbf{58} (2015), 402--414
                \mrev{3334936} \zbl{
1334.46043}
\bibitem{WeggeOlsen} \textsc{Wegge-Olsen, N.E.} { $K$-Theory and
$C^{*}$-Algebras,} \textit{Oxford Univ. Press, Oxford }(1993)
           \mrev{1222415 (95c:46116)} \zbl{ 0780.46038}   
\bibitem{WinterZacharias} \textsc{Winter, W.;  Zacharias, J.}  {Completely positive maps of order zero}, \textit{M\"unster J.  Math.} \textbf{2} (2009),  311--324
           \mrev{2545617 (2010g:46093)} \zbl{ 1190.46042}
\end{thebibliography}
\end{document}